\documentclass[12pt,reqno]{amsart}
\usepackage{color}
\usepackage[all]{xy}
\usepackage{amssymb,amsmath}

\usepackage{amsmath,amscd}
 
\usepackage{tikz}
\usetikzlibrary{matrix,arrows}

\usepackage{mathrsfs}

\usepackage{graphicx}
\usepackage[font={small,it},labelfont=bf]{caption}

\usepackage{setspace}

\date{\today}

\newtheorem{dummy}{anything}[section]
\newtheorem{theorem}[dummy]{Theorem}
\newtheorem*{thma}{Theorem A}
\newtheorem*{thmb}{Theorem B}

\newtheorem{lemma}[dummy]{Lemma}
\newtheorem{proposition}[dummy]{Proposition}
\newtheorem{corollary}[dummy]{Corollary}

\theoremstyle{definition}

  \newtheorem{example}[dummy]{Example}
  
  \newtheorem*{remark}{Remark}
   
\newcommand{\cA}{\mathcal A}
\newcommand{\cB}{\mathcal B}

\newcommand{\cD}{\mathcal D}

\newcommand{\cG}{\mathcal G}

\newcommand{\bZ}{\mathbb Z}

\newcommand{\bC}{\mathbb C}
\newcommand{\bR}{\mathbb R}

\newcommand{\vv}{\, | \,}

\providecommand{\abs}[1]{\lvert #1\rvert}

\DeclareMathOperator{\sign}{sign}

\DeclareMathOperator{\Ad}{Ad}

\DeclareMathOperator{\ch}{ch}

\DeclareMathOperator{\Tr}{Tr}

\DeclareMathOperator{\diag}{diag}

\newcommand{\gr}{\operatorname{gr}}

\begin{document}

\title{Equivariant rho-invariants and instanton \\ homology of torus knots}

\author{Nima Anvari}
\address{Department of Mathematics 
 \newline\indent
University of Miami
 \newline\indent
Coral Gables, FL}
\email{anvarin{@}math.miami.edu}

\begin{abstract}
The equivariant rho-invariants studied in this paper are a version of 
the classical rho-invariants of Atiyah, Patodi, and Singer in the 
presence of an isometric involution. We compute these rho-invariants 
for all involutions on the 3-dimensional lens spaces with $1$-dimensional 
fixed point sets, as well as for some involutions on Brieskorn homology 
spheres. As an application, we compute the generators and Floer gradings 
in the singular instanton chain complex of $(p,q)$-torus knots with odd 
$p$ and $q$.
\end{abstract}
\maketitle

\section{Introduction}
Let $Y$ be a closed oriented Riemannian $3$-manifold and $\alpha: \pi_1(Y) \rightarrow U(n)$ a unitary representation. Associated with this data are the classical $\eta$- and $\rho$-invariants of Atiyah, Patodi and Singer \cite{APS2} defined as follows. Consider the universal cover $\widetilde Y \to Y$ and the flat vector bundle $E_{\alpha} = \widetilde Y \times_{\pi_1 (Y)}\,\mathbb C^n$. The twisted odd signature operator $B^{ev}_{\alpha}: \Omega^{ev}(Y; E_{\alpha}) \to \Omega^{ev}(Y; E_{\alpha})$ given by
\[
B^{ev}_{\alpha}(\phi)=(-1)^p(\ast d_{\alpha}-d_{\alpha}\ast)(\phi),\quad \deg(\phi) = 2p,
\]
is self-adjoint and elliptic. As such it has discrete spectrum without accumulation points and with eigenvalues of finite multiplicity. Let $\lambda$ be the eigenvalues of the operator $B^{ev}_{\alpha}$ and let $W_{\lambda,\alpha}$ be the corresponding eigenspaces. The function 
\[
\eta_{\alpha} (Y)(s) = \sum_{\lambda \neq 0}\; \frac{\sign \lambda\cdot \dim W_{\lambda,\alpha}}{|\lambda|^s}
\]
admits a meromorphic continuation to the complex plane with no pole at the origin; the value of this function at $s = 0$ is the $\eta$-invariant $\eta_{\alpha} (Y)$. The $\rho$-invariant of $(Y,\alpha)$ is then defined as 
\[
\rho_{\alpha}(Y) = \eta_{\alpha} (Y) - \eta_{\theta} (Y),
\]
where $\theta: \pi_1 (Y) \rightarrow U(n)$ is the trivial representation. It is independent of the metric and defines a diffeomorphism invariant of $(Y,\alpha)$.

This paper deals with equivariant analogues of the $\eta$- and $\rho$-invariants in presence of an orientation preserving isometric involution $\tau: Y\rightarrow Y$ as defined in \cite{APS2} and \cite{Don78}. Suppose that a representation $\alpha: \pi_1(Y) \rightarrow U(n)$ is such that its pull-back $\tau^{\ast}\alpha$ is conjugate to $\alpha$. Then the pull-back of the flat vector bundle $E_{\alpha}$ is isomorphic to $E_{\alpha}$ hence $\tau$ can be lifted to a flat bundle automorphism $\nu: E_{\alpha} \rightarrow E_{\alpha}$. Note that the lift $\nu$ need not be an involution and that there may be more than one lift. The lift $\nu$ acts on $\Omega^{ev}(Y; E_{\alpha})$ via pull-back of forms making the diagram
\smallskip
\[
\begin{CD}
\Omega^{ev}(Y;E_{\alpha}) @> B^{ev}_{\alpha} >>  \Omega^{ev}(Y;E_{\alpha}) \\
@V \nu^{\ast} VV @VV \nu^{\ast}   V  \\
\Omega^{ev}(Y;E_{\alpha}) @> B^{ev}_{\alpha} >> \Omega^{ev}(Y;E_{\alpha}).
\end{CD}
\]

\smallskip\noindent
commute. Each eigenspace $W_{\lambda,\alpha}$ of the operator $B^{ev}_{\alpha}$ is then $\nu^{\ast}$-invariant, and the function
\[
\eta_{\alpha}(\nu,Y)(s)=\sum_{\lambda \neq 0}\; \dfrac{\sign \lambda\cdot \Tr(\nu^{\ast} \vv W_{\lambda,\alpha})}{\abs{\lambda}^s}
\]

\smallskip\noindent
admits a meromorphic continuation to the complex plane with no pole at the origin. The value of this function at $s = 0$ is the equivariant $\eta$-invariant of the triple $(Y,\alpha,\nu)$ which we denote by $\eta_{\alpha}(\nu,Y)$. 

Let $\theta: \pi_1 (Y) \to U(n)$ be the trivial representation and assume that the flat vector bundle $E_{\alpha}$ is abstractly trivial. Define the equivariant $\rho$-invariant by the formula
\[
\rho_{\alpha}(\nu,Y)=\eta_{\alpha}(\nu,Y)-\eta_{\theta}(\nu,Y),
\]
where the same lift $\nu$ is chosen for both $\alpha$ and $\theta$. For a fixed lift $\nu$, the equivariant $\rho$-invariant turns out to be independent of the metric and is invariant under diffeomorphims of $Y$ that commute with $\tau$.

In the case of a free involution, the equivariant $\rho$-invariant can often be computed using the Fourier transform techniques of \cite{APS2}. However, the condition on the involution to be free is too restrictive: for example, non-free involutions are intrinsic to the calculation of Floer indices in the instanton knot homology by Poudel and Saveliev \cite{poudel2015link}. Calculations for non-free involutions are rather sparse in the literature; examples of such calculations can be found in Degeratu \cite{degeratu2009eta} and Saveliev \cite{saveliev1999floer}. We add to this body of knowledge an explicit calculation of the equivariant $\rho$-invariants for smooth involutions on $3$-dimensional lens spaces $L(p,q)$ with one-dimensional fixed point sets, as classified by Hodgson and Rubinstein \cite{hodgson1985involutions}.

Let us view the lens space $L(p,q)$ as the quotient of the unit sphere $S^3 \subset \mathbb C^2$ by the action of the cyclic group of order $p$ whose generator sends $(z_1,z_2)$ to $(\zeta z_1, \zeta^q z_2)$ for $\zeta = e^{2\pi i/p}$. Consider the involution $\tau: L(p,q) \to L(p,q)$ defined on the orbits of this action by the formula $\tau ([z_1,z_2]) = [-z_1, z_2]$; this is one of the involutions on the Hodgson--Rubinstein list. It has a circle as its fixed point set. With a base point on this circle, the induced action $\tau_{\ast}$ on the fundamental group $\pi_1 (L(p,q))$ is trivial, therefore, any representation $\alpha: \pi_1(L(p,q)) \rightarrow U(1)$ satisfies $\tau^{\ast}\alpha = \alpha$. The formula $\nu ([z_1,z_2, \xi]) = [-z_1,z_2, \xi]$ then defines a lift $\nu: E_{\alpha} \to E_{\alpha}$ which will be called the \emph{canonical lift} and denoted by $\tau$. Despite the fact that the bundle $E_{\alpha}$ need not be trivial, the equivariant $\rho$-invariant is well-defined using the canonical lift on both bundles. The same can be said about the involution on $L(p,q)$ defined by $\tau ([z_1,z_2]) = [z_1, -z_2]$.

\begin{thma}
Let $\alpha:\pi_1(L(p,q))\rightarrow U(1)$ be a representation sending the canonical generator to $e^{2\pi i \ell/p}$, with $p$ odd. If the involution $\tau: L(p,q) \to L(p,q)$ is given by $\tau([z_1,z_2]) = [-z_1, z_2]$ then
\smallskip
\[
\rho_{\alpha}(\tau,L(p,q))=-\dfrac{2}{p}\;\sum_{k=1}^{p-1}\; \tan \left(\dfrac{\pi k}{p}\right)\cot\left(\dfrac{\pi k q}{p}\right)\sin^2\left(\dfrac{\pi k \ell}{p}\right).
\]

\medskip\noindent
If the involution $\tau: L(p,q) \to L(p,q)$ is given by the formula $\tau ([z_1,z_2]) = [z_1,-z_2]$ then 
\smallskip
\[
\rho_{\alpha}(\tau,L(p,q))=-\dfrac{2}{p}\;\sum_{k=1}^{p-1}\; \cot\left(\dfrac{\pi k}{p}\right)\tan\left(\dfrac{\pi q k}{p}\right)\sin^2\left(\dfrac{\pi k \ell}{p}\right).
\]
If $p \geq 2$ is even, the formulas above are valid with the modification that the singular terms be left out of the summation.
\end{thma}

\smallskip

It follows from Hodgson and Rubinstein \cite{hodgson1985involutions} that every smooth involution on a lens space $L(p,q)$ which has one-dimensional fixed point set and satisfies $\tau^{\ast}\alpha = \alpha$ for all representations $\alpha: \pi_1(L(p,q)) \to U(1)$ is strongly equivalent to one of the involutions in the above theorem. The only other involution in their classification for which $\tau^{\ast}\alpha$ and $\alpha$ are conjugate, this time as $SU(2)$ representations, is induced by the complex conjugation on $S^3$. In this case, it was shown by Saveliev \cite{saveliev1999floer} that the equivariant $\rho$-invariants vanish for all $L(p,q)$. We will give an independent proof of this fact in Section \ref{S:complex}. In the special case of $p=2$, every orientation preserving involution with one-dimensional fixed point set is strongly equivalent to the complex conjugation involution on $\mathbb{R}P^3$ with the fixed point set projecting to the Hopf link, and the equivariant $\rho$-invariants for such an involution all vanish.  

The trigonometric sums that show up in Theorem A are further discussed in Section \ref{S:sums} where we interpret them as a signed count of certain lattice points in the plane, similar to that of Casson and Gordon \cite{casson1986cobordism}.

This paper was inspired by the formula of Poudel and Saveliev \cite{poudel2015link} for computing Floer gradings in the singular instanton knot homology $I^{\natural}(K)$ of Kronheimer and Mrowka \cite{kronheimer2011khovanov}. In Section \ref{S:floer} we use their formula to calculate the Floer gradings in $I^{\natural}(K)$ for all $(p,q)$-torus knots $K$ with odd $p$ and $q$. The main ingredient in that calculation is an extension of Theorem A to a class of Brieskorn homology spheres as we now explain.

Let $p$ and $q$ be positive relatively prime odd integers. The Brieskorn homology sphere $\Sigma(2,p,q)$ is the link of the complex surface singularity $\{(x,y,z)\in \bC^3 \vv x^2+y^p+z^q=0 \}\,\cap\,S^5$ in $\bC^3$. It has a natural orientation preserving involution $\tau (x,y,z) = (-x,y,z)$ contained in the circle action that gives $\Sigma(2,p,q)$ the structure of a Seifert fibered manifold. This involution turns $\Sigma(2,p,q)$ into the double branched cover of $S^3$ with branching set the right-handed $(p,q)$-torus knot. Since $\Sigma(2,p,q)$ is an integral homology sphere, all non-trivial representations $\alpha: \pi_1 (\Sigma(2,p,q)) \to SU(2)$ are irreducible. For every irreducible representation $\alpha$, its pull back $\tau^*\alpha$ is conjugate to $\alpha$, and there is an essentially unique lift $\nu: E_{\alpha} \to E_{\alpha}$ that does the job. We will derive a formula for the equivariant $\rho$-invariant of the adjoint representation $\Ad\alpha: \Sigma (2,p,q) \to SO(3)$ with respect to the adjoint lift, which we again call $\nu$. To conform to our general setup, we will view $SO(3)$ as a natural subgroup of $SU(3)$.

To state our result, choose a set of Seifert invariants $(0;(2,b_1),(p,b_2),q,b_3))$ for $\Sigma(2,p,q)$, where $b_2$ and $b_3$ can be any even integers such that $pqb_1 + 2q b_2 + 2p b_3 = 1$. In addition, recall from Fintushel--Stern \cite{FS90} that an irreducible representation $\alpha: \pi_1 (\Sigma(2,p,q)) \to SU(2)$ is uniquely determined by its rotation numbers $(1,\ell_2,\ell_3)$, where $\ell_2$ and $\ell_3$ are even integers satisfying $0<\ell_2< p$ and $0<\ell_3<q$ and certain additional constraints.

\begin{thmb}
Let $p$, $q$ be relatively prime odd integers and $\alpha: \pi_1 (\Sigma(2,p,q)) \to SU(2)$ an irreducible $SU(2)$ representation with rotation numbers $(1,\ell_2,\ell_3)$. Then the equivariant $\rho$-invariant of $\Sigma(2,p,q)$ with respect to the aforementioned involution $\tau$ and the lift $\nu$ is given by 
\begin{align*}
\rho_{\Ad\alpha}(\nu,\Sigma(2,p,q))=1-\dfrac{2}{pq} &-\dfrac{4}{p}\;\sum_{k=1}^{p-1}\cot\left(\dfrac{\pi k}{p} \right)\tan\left( \dfrac{\pi b_2 k}{p}\right) \cos^{2}\left(\dfrac{\pi k \ell_2}{p}\right) \\&-\dfrac{4}{q}\;\sum_{k=1}^{q-1}\cot\left(\dfrac{\pi k}{q}\right)\tan\left(\dfrac{\pi b_3 k}{q}\right)\cos^{2}\left(\dfrac{\pi k \ell_3}{q}\right).
\end{align*}
\end{thmb}

\medskip\noindent
\textbf{Acknowledgements.} The author thanks Nikolai Saveliev for sharing his expertise and is grateful to him for many invaluable discussions and suggestions on the earlier draft of this paper. A suggestion of John Lott communicated to the author by Nikolai Saveliev has been useful in some of the computations. The author also thanks Ken Baker for many stimulating conversations. 


\section{Preliminaries}
In this section we collected some well-established results which will be used later in the paper. 


\subsection{Bundle automorphisms}
Let $Y$ a closed oriented $3$-manifold with a smooth orientation preserving involution $\tau: Y \to Y$ generating a subgroup $\bZ/2$ of the group of diffeomorphisms of $Y$. Given a vector bundle $E \to Y$ such that $\tau^*E$ is isomorphic to $E$, we will follow Austin \cite{Austin} and consider the group $\widehat{\cG}$ of bundle automorphisms of $E$ lifting the elements of $\bZ/2$. We obviously have a short exact sequence
\[
1\longrightarrow \mathcal{G} \longrightarrow \mathcal{\widehat{G}}\longrightarrow \bZ/2 \longrightarrow 1,
\]
where $\cG$ is the group of gauge transformations of $E$. The group $\widehat{\cG}$ acts on the affine space $\cA$ of connections on $E$ via $g^* A = g\cdot A \cdot g^{-1}$ giving rise to a well-defined action $\tau^{\ast}: \cB \to \cB$ on the configuration space $\cB = \cA/\cG$. We are interested in the fixed point set of this action.

Let $A \in \cA$ be a connection whose gauge equivalence class is fixed by $\tau^{\ast}$ then there is a lift $\nu \in \widehat{\cG}$ such that $\nu^* A = A$. All such lifts generate the stabilizer $\widehat{\cG}_A$ of $A$ in the group $\widehat{\cG}$, and we have a short exact sequence 
\begin{align}\label{exact sequence of stabilizers}
1\longrightarrow \mathcal{G}_A \longrightarrow \mathcal{\widehat{G}}_A\longrightarrow \bZ/2 \longrightarrow 1,
\end{align}
where $\cG_A$ is the stabilizer of $A$ in the gauge group $\cG$. Replacing $A$ by a gauge equivalent connection results in conjugating this exact sequence by the gauge transformation. The following special cases of this construction will be of particular importance in this paper.

\begin{example}
Let $E$ be a real vector bundle of rank three and $A$ an irreducible $SO(3)$ connection. The stabilizer $\cG_A$ is then trivial and it follows from \eqref{exact sequence of stabilizers} that there exists a unique lift $\nu$ leaving $A$ invariant. This lift has the property that $\nu^2 = 1$. 
\end{example}

\begin{example}
Let $E$ be a complex vector bundle of rank two and $A$ an irreducible $SU(2)$ connection. Then $\cG_A = \{\pm 1\}$ and the lifts fixing $A$ satisfy $\nu^2 = \pm 1$ with the sign depending on whether $\widehat{\cG}_A$ is isomorphic to $\bZ/2\,\oplus\,\bZ/2$ or $\bZ/4$. Note that since every $SU(2)$-bundle $E$ over $Y$ is trivial, we may fix a trivialization $E = Y \times \bC^2$ and view our lifts as functions $\nu: Y \times \bC^2 \rightarrow Y \times \bC^2$ given by $\nu(y,\xi) = (\tau(y), u(y) \xi)$ with $u: Y \rightarrow SU(2)$. According to \cite[page 38]{ruberman_saveliev_2004rohlin}), every lift $\nu$ is then equivalent up to gauge transformation to a lift with the constant $u: Y\rightarrow SU(2)$.
\end{example}

\begin{example}
Let $\alpha: \pi_1(Y) \rightarrow U(1)$ be a representation such that $\tau^{\ast}\alpha$ is conjugate to $\alpha$ and choose $u \in U(1)$ with $\tau^{\ast}\alpha=u^{-1} \alpha u = \alpha$. This choice of $u$ determines a lift $\nu: E_{\alpha} \to E_{\alpha}$ as follows. Lift $\tau$ to a unique involution $\widetilde{\tau}:\widetilde Y \rightarrow \widetilde Y$ on the universal cover of $Y$ and define $\widetilde{\nu}: \widetilde Y \times \mathbb C \to \widetilde Y \times \mathbb C$ by the formula 
\[
\widetilde{\nu}(y,\xi)=(\widetilde{\tau}(y),\, u\cdot \xi).
\]
Note that $\widetilde \nu$ preserves the product connection on $\widetilde Y \times \mathbb C$. The flat bundle $E_{\alpha}$ is the quotient of $\widetilde Y \times \mathbb C$ by the action of $\pi_1(Y)$ given by the formula 
\[
\gamma(y,\xi)=(y\cdot\gamma,\, \alpha(\gamma^{-1})\cdot\xi).
\]
Since $\widetilde \nu$ obviously commutes with this action it descends to a well defined map $\nu: E_{\alpha} \to E_{\alpha}$ lifting $\tau$. The lift $\nu$ will be referred to as a \emph{constant lift}; the canonical lift of the introduction is then the constant lift corresponding to the choice of $u = 1$. Note that a constant lift is uniquely determined by its action on the fiber of $E_{\alpha}$ over any fixed point of $\tau: Y \to Y$. Also note that any lift $\nu: E_{\alpha} \to E_{\alpha}$ whose induced lift $\widetilde\nu: \widetilde Y\times \mathbb C \rightarrow \widetilde Y \times \mathbb C$ preserves the product connection must be a constant lift, therefore, any lift $\nu: E_{\alpha} \to E_{\alpha}$ is gauge equivalent to a constant lift.
\end{example}

It is not difficult to see that conjugating a representation $\alpha: \pi_1 (Y) \to U(n)$ replaces the operator $B^{ev}_{\alpha}$ by a conjugate and hence does not change the values of $\eta_{\alpha} (Y)$ and $\rho_{\alpha} (Y)$. Similarly, replacing a lift $\nu:E_{\alpha} \to E_{\alpha}$ by a gauge equivalent lift does not change the values of $\eta_{\alpha} (\nu,Y)$ and $\rho_{\alpha} (\nu,Y)$. 
\subsection{Index theorems}
The classical $\eta$-invariant of the odd signature operator arises naturally as a correction term in the index formula for the signature operator on a manifold with boundary. Let $Y$ be the boundary of a smooth compact oriented $4$-manifold $X$ and suppose that a representation $\pi_1(Y) \rightarrow U(n)$ extends to a representation $\pi_1(X) \to U(n)$ denoted again by $\alpha$. Then we have a flat bundle $E_{\alpha}$ over $X$ whose restriction to the boundary is the given flat bundle over $Y$. For a Riemannian metric on $X$ which restricts to a product metric near the boundary $Y$, the twisted signature \cite{APS2} is given by the formula
\begin{align*}
\sign_{\alpha}(X)=\int_{X} \ch(E_{\alpha})\,\mathcal{L}-\eta_{\alpha}(Y)
\end{align*}  
where $\mathcal{L}=p_1(X)/3$ and $p_1(X)$ is the first Pontryagin class of the tangent bundle of $X$. Since the bundle $E_{\alpha}$ is flat, the Chern character $\ch(E_{\alpha})$ equals $n$, the rank of $E_{\alpha}$.

Let $\tau: X \to X$ be an orientation preserving isometric involution extending the involution $\tau: Y \to Y$ on the boundary $Y$ and suppose that $\alpha: \pi_1(X) \rightarrow U(n)$ is such that $\tau^{\ast}\alpha$ is conjugate to $\alpha$. Given an extension of the lift $\nu: E_{\alpha} \to E_{\alpha}$ to the flat bundle over $X$, we have the following formula 
\begin{align*}
\sign_{\alpha}(\nu,X)=\int_{X^{\tau}} \ch_{\nu}(E_{\alpha})\,\mathcal{L}(X^{\tau})-\eta_{\alpha}(\nu,Y),
\end{align*}
see Donnelly \cite{Don78}. In general, the fixed set $X^{\tau}$ of the involution $\tau$ is a disjoint union of isolated fixed points and $2$-dimensional fixed surfaces, possibly with non-empty intersection with the boundary. Since the bundle $E_{\alpha}$ is flat, the equivariant Chern character $\ch_{\nu}(E_{\alpha})$ is just the trace of the lift $\nu$ acting on the fibers over the fixed point set. The class $\mathcal{L}(X^{\tau})$ is a certain combination of characteristic classes; in the special case when $X^{\tau}$ consists of only $2$-dimensional components, $\mathcal{L}(X^{\tau}) = e(N(X^{\tau}))$, the Euler class of the normal bundle to the fixed point set.

Choose $\alpha$ to be the trivial representation $\theta: \pi_1 (X) \to U(n)$ and use the same lift $\nu$ to obtain the formula
\begin{align*}
\sign_{\theta}(\tau,X)=\int_{X^{\tau}} \ch_{\nu}(E_{\theta})\,\mathcal{L}(X^{\tau})-\eta_{\theta}(\nu,Y).
\end{align*}
The difference of the last two displayed formulas then results in
\begin{align*}
\rho_{\alpha}(\nu,Y)=\sign_{\theta}(\nu,X)-\sign_{\alpha}(\nu,X).
\end{align*}
Another application of the above formula to the product $Y\times [0,1]$ as in \cite{APS2} yields the following results.

\begin{theorem} 
For a choice of equivariant representation $\alpha$ and a lift $\nu$, the invariant $\rho_{\alpha}(\nu,Y)$ is independent of the choice of $\tau$-invariant metric and is a diffeomorphism invariant of $(Y,\alpha,\tau)$, i.e. it is invariant under diffeomorphisms of $Y$ that commute with the involution $\tau$. 
\end{theorem}


\subsection{Involutions on lens spaces}
The classification of smooth involutions on lens spaces with one-dimensional fixed point sets was given by Hodgson and Rubinstein \cite{hodgson1985involutions}. In this section we summarize some of their main results that will be used in this paper. 

Let $V=S^{1}\times D^{2} \subset \mathbb C^2$ denote a solid torus. Then every orientation preserving involution on $V$ is strongly equivalent to one of the following\,:
\begin{itemize}
\item \; $g_1 (t,z) = (\bar{t},\bar{z})$ 
\item \; $g_2 (t,z) = (t,-z)$ 
\item \; $g_3 (t,z) = (-t,z)$ 
\item \; $g_4 (t,z) = (-t,-z)$ 
\end{itemize}  
The involution $g_1$ fixes two properly embedded arcs in $V$ with orbit space the $3$-ball $B^3$. The fixed arcs in $V$ project to properly embedded unknotted arcs in the orbit space. The involution $g_2$ has a fixed core as the fixed point set and the quotient $V/g_2$ is again a solid torus whose core is the image of the fixed set. The involutions $g_3$ and $g_4$ are free and equivalent but not strongly equivalent. The orbit spaces are again solid tori. 

We next describe smooth orientation preserving involutions $\tau$ on lens spaces $Y=L(p,q)$ with non-empty fixed point sets. Let $Y=V \cup V^{\prime}$ denote a Heegaard splitting consisting of $\tau$-invariant solid tori $V$ and $V^{\prime}$. Then $\tau$ is strongly equivalent to one of the following\,:

\bigskip\noindent\textbf{Type A.}\; If $\tau \vv V$ is of type $g_1$ then $\tau \vv V^{\prime}$ is also of type $g_1$. The orbit space $Y/\tau$ is the $3$-sphere $S^3$. The fixed point set consists of two circles if $p$ is even and one circle if $p$ is odd. The circles project to a link or knot in $S^3$. This involution can be given explicitly by $\tau ([z_1,z_2]) = [\bar{z_1},\bar{z_2}]$.

\bigskip\noindent\textbf{Type B.}\; If $\tau \vv V$ is of type $g_2$, then there are two cases. If $p$ is even then $\tau \vv V^{\prime}$ is also of type $g_2$. The fixed point set consists of two core circles in each solid tori, and the orbit space is $L(p/2,q)$. If $p$ is odd, then $\tau \vv V^{\prime}$ is a free involution of type $g_3$ or $g_4$. The fixed point set is the core circle in $V$ and the orbit space is $L(p,2q)$. This involution can be given explicitly by $\tau ([z_1,z_2]) = [z_1,-z_2]$.

\bigskip\noindent\textbf{Type $\mathbf B^{\prime}$.} If $\tau \vv V'$ is of type $g_2$ then there are two cases. If $p$ is even then this involution is strongly equivalent to $\textbf{B}$. If $p$ is odd, then $\tau \vv V$ is of type $g_3$ or $g_4$. The fixed point set in the core circle in $V$ and the orbit space is $L(p,2q^{\ast})$ where $qq^{\ast}\equiv 1 \pmod p$. This involution can be given explicitly by $\tau ([z_1,z_2]) = [-z_1,z_2]$.

\bigskip\noindent\textbf{Types C and $\mathbf C^{\prime}$.} If the restriction of $\tau$ to $\partial V=\partial V^{\prime}$ is orientation-reversing with non-empty fixed set, then it is strongly equivalent to a linear involution with matrix
\begin{align*}
\pm
\begin{pmatrix}
-q & p \\ p^{\prime} & q
\end{pmatrix}
\end{align*}
where $p^{\prime}$ satisfies $pp^{\prime}+q^2=1$. This lead to two more types of involutions.

\bigskip The following theorem is the main classification result of Hodgson and Rubinstein.

\begin{theorem}\cite[p.83]{hodgson1985involutions}
Every smooth orientation preserving involution on a lens space $L(p,q)$ with fixed points is strongly equivalent to one of the involutions $A$, $B$, $B^{\prime}$, $C$, $C^{\prime}$. If $p=2$ then all of these involutions are strongly equivalent to each other.
\end{theorem}

We conclude this section by describing the actions induced by the above involutions on $\pi_1 (L(p,q)) = \mathbb Z/p$ with a base point in the fixed point set. The involution of type A acts via multiplication by $-1$. The involutions of types $B$ and $B^{\prime}$ act via multiplication by $+1$. The involutions of types $C$ and $C^{\prime}$ acts via multiplication by $-q$ and $q$, respectively. Therefore, all representations $\pi_1(L(p,q)) \to U(1)$ are equivariant with respect to the involutions of types $B$ and $B^{\prime}$, all representations $\pi_1(L(p,q)) \to SU(2)$ are equivariant with respect to the involution of types $A$, and no unitary representations of $\pi_1(L(p,q))$ are equivariant with respect to the involutions of types $C$ and $C^{\prime}$ unless $q = 1\pmod p$. However, in the latter case the involutions $C$ and $C^{\prime}$ are strongly equivalent to involutions of either type $A$ or type $B=B^{\prime}$. Since the equivariant $\rho$-invariants are only defined for equivariant representations, we will disregard involutions of types $C$ and $C^{\prime}$.


\section{Calculations for lens spaces}
This section is dedicated to the proof of Theorem A. The lens space $L(p,q)$ will be denoted by $Y$, its universal covering space by $\widetilde Y = S^3$, and the cyclic group $\pi_1 (Y)$ by $G$. Note that $G$ has a canonical generator, and denote by $\alpha_{\ell}: G \rightarrow U(1)$ the representation that sends that generator to $e^{2\pi i \ell/p}$, $\ell=0,\cdots,p-1$. We will begin by relating the equivariant $\eta$-invariant of the odd signature operator $\widetilde{B}^{ev}: \Omega^{ev}(\widetilde Y) \to \Omega^{ev} (\widetilde Y)$ to the twisted equivariant $\eta$-invariants of $B^{ev}_{\alpha_{\ell}}: \Omega^{ev}(Y;E_{\alpha_{\ell}}) \to \Omega^{ev}(Y;E_{\alpha_{\ell}})$. We will use trivial bundle lifts throughout.

\begin{proposition} 
Let $\tau: Y \to Y$ be an isometric involution which lifts to an involution $\widetilde{\tau}: S^3 \to S^3$. Then
\[
\eta(\widetilde{\tau}g,S^3)\; =\; \sum_{\ell=0}^{p-1}\; \eta_{\alpha_{\ell}}(\tau,Y)\cdot\overline{\alpha}_{\ell}(g)\quad\text{for any $g \in G$}.
\]
\end{proposition}

\begin{proof}
The equivariant $\eta$-invariant is defined as the value at $s = 0$ of the meromorphic continuation to the complex plane of the function 
\[
\eta(\widetilde{\tau}g,S^3)(s) = \sum_{\lambda \neq 0}\; \dfrac{\sign \lambda \cdot \Tr(\widetilde{\tau}^{\ast}g^{\ast}\vv W_{\lambda})}{\abs{\lambda}^{s}}\,,
\]
where $W_{\lambda}$ is the $\lambda$-eigenspace of the operator $B^{ev}:\Omega (S^3) \rightarrow \Omega(S^3)$. Each $W_{\lambda}$ is a $\bZ[G]$-module which can be decomposed into irreducible submodules
\[
W_{\lambda} = \bigoplus_{\ell=0}^{p-1}\; W_{\lambda,\alpha_{\ell}}
\]
so that 
\[
\Tr(\widetilde{\tau}^{\ast}g^{\ast}\vv W_{\lambda})=\sum_{\ell=0}^{p-1}\; \Tr(\widetilde{\tau}^{\ast}g^{\ast}\vv W_{\lambda,\alpha_{\ell}}).
\]
On each of the modules $W_{\lambda,\alpha_{\ell}}$ the action of $g^*$ is diagonal of weight $g^{-\ell}$ hence
\[
\Tr(\widetilde{\tau}^{\ast}g^{\ast}\vv W_{\lambda,\alpha_{\ell}}) = \Tr(\widetilde{\tau}^{\ast}\vv W_{\lambda,\alpha_{\ell}})\cdot g^{-\ell} = \Tr(\tau^{\ast}\vv W_{\lambda,\alpha_{\ell}})\cdot\overline{\alpha}_{\ell}(g),
\]
where we used the natural identification of $W_{\lambda,\alpha_{\ell}}$ with the $\lambda$-eigenspace of the operator $B^{ev}_{\alpha_{\ell}}$. From this we easily conclude that
\[
\eta(\widetilde{\tau}g,S^3)(s) = 
\sum_{\ell=0}^{p-1}\; \eta_{\alpha_{\ell}}(\tau,Y)(s)\cdot \overline{\alpha}_{\ell}(g)
\]
for all $s$ whose real part is sufficiently large. The result now follows from the uniqueness of meromorphic continuation.
\end{proof}

\begin{corollary}\label{C:rho}
Let $\alpha: G \to U(1)$ be an arbitrary unitary representation and $\chi_{\alpha}: G \to U(1)$ its character. Then
\begin{equation}\label{E:rho}
\rho_{\alpha}(\tau,Y)=\dfrac{1}{\abs{G}}\;\sum_{g \neq 1}\;\eta(\widetilde{\tau}g,S^3)\cdot (\chi_{\alpha}(g)-\dim(\alpha)).
\end{equation}
\end{corollary}

\begin{proof}
This easily follows from the definition of the equivariant $\rho$-invariant and the formula 
\[
\eta_{\alpha}(\tau,Y)=\dfrac{1}{\abs{G}}\;\sum_{g\in G}\;\eta(\widetilde{\tau}g,S^3)\cdot\chi_{\alpha}(g)
\]
obtained by applying the Fourier transform to the formula in the above proposition.
\end{proof}


\subsection{Involutions of types $B$ and $B^{\prime}$}
Let $\tau$ be the involution $\tau ([z_1,z_2]) = [-z_1,z_2]$. Write the elements of $G$ in the form $g^k$ with $k = 0,\ldots,p-1$ using the canonical generator $g \in \pi_1 (L(p,q))$ then the formula \eqref{E:rho} gives
\[
\rho_{\alpha}(\tau,Y)=\dfrac{1}{p}\;\sum_{k=1}^{p-1}\; \eta(\widetilde{\tau} g^k,S^3)\cdot(e^{2 \pi i k \ell/p}-1).
\]
For any $k \ne 0$, the map $\widetilde{\tau} g^k$ has an isolated fixed point $(0,0)$ with the rotation number $(p+2k, 2kq) \pmod {2p}$ hence
\[
\eta(\widetilde{\tau} g^k,S^3)=-\cot\left(\dfrac{\pi}{2}+\dfrac{\pi k}{p}\right)\cot\left(\dfrac{\pi k q}{p}\right),
\]
see \cite{APS2}, where it is referred to as the signature defect. The above formula for the $\rho$-invariant then becomes
\[
\rho_{\alpha}(\tau,Y) = \dfrac{1}{p}\;\sum_{k=1}^{p-1}\; \tan\left(\dfrac{\pi k}{p}\right)\cot\left(\dfrac{\pi k q}{p}\right)\left(e^{2\pi i k\ell/p}-1\right).
\]
That this formula matches that of Theorem A can be easily seen by re-writing
\begin{multline}\notag
\rho_{\alpha}(\tau,Y)=\dfrac{1}{2p}\;\sum_{k=1}^{p-1}\; \tan\left(\dfrac{\pi k}{p}\right)\cot\left(\dfrac{\pi k q}{p}\right)\left(e^{2\pi ik \ell/p}-1\right)  \\
+\dfrac{1}{2p}\;\sum_{k=1}^{p-1}\; \tan\left(\dfrac{\pi (p-k)}{p}\right)\cot\left(\dfrac{\pi (p-k) q}{p}\right)\left(e^{2\pi i(p-k) \ell/p}-1\right)
\end{multline}
and then simplifying to
\begin{multline}\notag
\rho_{\alpha}(\tau,Y)=\dfrac{1}{2p}\;\sum_{k=1}^{p-1}\; \tan\left(\dfrac{\pi k}{p}\right)\cot\left(\dfrac{\pi k q}{p}\right)\left(e^{2\pi ik \ell/p}+e^{-2\pi ik \ell/p} - 2\right)  \\
=-\dfrac{2}{p}\;\sum_{k=1}^{p-1}\; \tan\left(\dfrac{\pi k}{p}\right)\cot\left(\dfrac{\pi k q}{p}\right)\sin^2\left(\dfrac{\pi k \ell}{p}\right).
\end{multline}
The other involution of Theorem A, given by the formula $\tau ([z_1,z_2])= [z_1, -z_2]$, is handled similarly: using the rotation number $(2k,p+2qk)\pmod {2p}$ we first obtain the answer
\begin{align*}
\rho_{\alpha}(\tau,Y)=\dfrac{1}{p}\;\sum_{k=1}^{p-1}\; \cot\left(\dfrac{\pi k}{p}\right)\tan\left(\dfrac{\pi k q}{p}\right)\left(e^{2\pi i k \ell/p}-1\right)
\end{align*}
and then simplify it to arrive at the formula of the theorem. 
\vspace{-0.20 cm} 
\subsection{Involution of type $A$}\label{S:complex}
The fixed point set of the involution $\widetilde\tau: D^4 \to D^4$ given by $\widetilde{\tau} (z_1,z_2) = (\overline z_1, \overline z_2)$ is a properly embedded disk $F \subset D^4$ with the boundary in $S^3$. This involution defines an involution on $S^3$ which descends to an involution $\tau$ on the lens space $Y = L(p,q)$ referred to as an involution of type $A$. Let $\alpha: \pi_1 (L(p,q)) \to SU(2)$ be a representation sending the canonical generator to 
\begin{align*}
\begin{pmatrix}
e^{2\pi i \ell /p} & 0 \\ 0 & e^{-2\pi i \ell /p} \\
\end{pmatrix}.
\end{align*}
Since $\tau$ acts on $\pi_1(L(p,q))$ via multiplication by $-1$, the representations $\tau^{\ast}(\alpha)$ and $\alpha$ are conjugate in $SU(2)$. Therefore, we have a well defined equivariant $\rho$-invariant $\rho_{\alpha}(\nu,L(p,q))$ for a certain lift $\nu: E_{\alpha} \to E_{\alpha}$.  According to Saveliev \cite{saveliev1999floer} this invariant vanishes for all $\alpha$; we offer here an alternative proof of that statement.

The $G$-signature theorem of Donnelly \cite{Don78} applied to $\widetilde{\tau}g$ in the extension
\[
1 \rightarrow \bZ/p \rightarrow D_{2p} \rightarrow \bZ/2\rightarrow 1
\]
where $\widetilde{\tau}$ is the complex conjugation involution on the $4$-ball yields the formula 
\[
0=\sign(\widetilde{\tau} g, D^4)=\int_{F} e(N(F)) -\eta(\widetilde{\tau} g, S^3).
\]
The two terms of the right hand side of this formula are metric dependent. Once we show that the integral term vanishes for a large class of metrics, the vanishing result for $\rho_{\alpha}(\nu, L(p,q))$ will follow from Corollary \ref{C:rho}.

To see that the integral vanishes, we attach an equivariant $2$-handle along the circle $\partial F = F\,\cap\,S^3$ with framing $0$ and choose a metric on $D^4$ that is product near the boundary and which restricts to the flat Euclidean metric on the $D^2$-normal disk bundle of $F$. In addition, choose the product of flat Euclidean metrics on the $2$-handle $D^2 \times D^2$. The result is a Riemannian $4$-manifold with an involution that has a $2$-sphere as its fixed point set. The vanishing of the integral in question is now a special case of the following lemma. 

\begin{lemma}
Let $(X_0,\tau)$ denote a smooth, compact $4$-manifold with boundary $S^3$ which admits an involution that has fixed point set $F$ of codimension two with non-empty intersection with the boundary. Choose a metric that is a product metric on the end and restricts to the flat Euclidean metric on the $D^2$ normal bundle of $F$. Then 
\begin{align*}
\int_F e(N(F))=0
\end{align*}
\end{lemma}
\begin{proof}
Attach an equivariant $2$-handle $(D^2 \times D^2,\tau)$ with product Euclidean metric to $S^3$ with framing $0$ along $\partial F$. The result is a compact $4$-manifold with boundary $X$ and a closed surface $\overline{F}$ as its fixed point set and empty intersection with the boundary $\partial X$. The self intersection $\overline{F}$ can be computed as a topological invariant and we do this by taking a perturbation of core disk $D^{\prime}$ in the $2$-handle such that $D^{\prime}\,\cap\, S^3=k^{\prime}$ and a perturbation $F^{\prime}$ with $\partial F^{\prime}=k^{\prime}$. Then 
\begin{align*}
\overline{F}\cdot \overline{F}=\operatorname{lk}(\partial F, k^{\prime})=0 
\end{align*} 
because of the framing of the $2$-handle. On the other hand,
\begin{align*}
\overline{F}\cdot \overline{F}=\int_{\overline{F}} e(N(\overline{F}))=\int_F e(N(F)),
\end{align*}
since the integral over the $2$-handle vanishes. 
\end{proof}

\begin{remark}
We mention here the case when $p\geq2$ is even in Theorem A. The same formulas work if we ignore the singular terms of the summation. This is because the fixed point set of $\widetilde{\tau}g^{k}$ is no longer an isolated fixed point but a fixed $2$-disk when $k=p/2$ and by the lemma above the term $\eta(\widetilde{\tau}g^k,S^3)$ vanishes. Similarly for the other involution given by $\tau ([z_1,z_2])= [z_1, -z_2]$, the fixed point set is a fixed $2$-disk and $\eta(\widetilde{\tau}g^k,S^3)$ vanishes for $k=q^{\ast}p/2$, where $q^{\ast}$ is the multiplicative inverse of $q$ mod $p$.
\end{remark}
\section{Combinatorial interpretation}\label{S:sums}
In this section we interpret the equivariant $\rho$-invariants of Theorem A as a signed count of certain lattice points in the plane, along the lines of Casson-Gordon \cite{casson1986cobordism}, see also \cite{anderson1995lift}. Let $p$ and $q$ be relatively prime integers and $\ell$ an even integer. Define the function
\[
\delta^{\tau}(p;q,\ell)=\dfrac{2}{p}\;\sum_{k=1}^{p-1}\;\cot\left(\dfrac{\pi k}{p}\right)\cot\left(\dfrac{\pi q k}{p}+\dfrac{\pi}{2}\right)\sin^2\left(\dfrac{\pi k \ell}{p}\right),
\]
which should be viewed as a modification the Casson--Gordon \cite{casson1986cobordism} function
\[
\delta(p;q,\ell)=\dfrac{2}{p}\;\sum_{k=1}^{p-1}\;\cot\left(\dfrac{\pi k}{p}\right)\cot\left(\dfrac{\pi q k}{p}\right)\sin^2\left(\dfrac{\pi k \ell}{p}\right).
\]
According to Theorem A, the function $\delta^{\tau}(p;q,\ell)$ equals the equivariant $\rho$-invariant $\rho_{\alpha}(\tau,L(p,q))$ corresponding to the canonical lift of  the involution $\tau([z_1,z_2]) = (z_1,-z_2)$ and the representation $\alpha$ sending the canonical generator of $\pi_1 (L(p,q))$ to $e^{2\pi i \ell/p}$. A straightforward calculation yields
\begin{multline}\notag
\delta^{\tau}(p;q,\ell)=\dfrac{1}{2p}\; \sum_{k=1}^{p-1} \left(\dfrac{e^{2\pi ik/p}+1}{e^{2\pi ik/p} -1} \right)\left(\dfrac{e^{2\pi iqk/p}-1}{e^{2\pi iqk/p}+1}\right) e^{-2\pi i k \ell/p}\,(e^{2\pi i k \ell/p}-1)^2\\
=\dfrac{1}{2p}\sum_{t^p=1,t\neq 1} \left(\dfrac{t+1}{t-1}\right) \left(\dfrac{t^q-1}{t^q+1} \right) t^{-\ell}(t^{\ell}-1)^2.
\end{multline}
This can be further simplified using the identities 
\[
t^{-\ell}(t+1)\left(\dfrac{t^{\ell}-1}{t-1} \right)=t^{-\ell}+2t^{1-\ell}+2t^{2-\ell}+\cdots+2t^{-1}+1
\]
and 
\[
\;\, (t^q-1)\left(\dfrac{t^{\ell}-1}{t^q+1}\right)= t^{\ell}-2t^{q}+2t^{2q}-\cdots-2t^{q(q^{\ast}\ell-1)}+1,
\]
where $1 \le q^* \le p-1$ is to be found from the equation $qq^* = 1 \pmod p$. To verify the latter identity, write
\begin{align*}
\dfrac{t^{\ell}-1}{t^{q}+1}=\dfrac{u^{q^{\ast}\ell}-1}{u+1}=-(1-u+u^2-\cdots -u^{q^{\ast}\ell-1})
\end{align*}
with respect to the variable $u = t^q$ (remember that $\ell$ is assumed to be even) and multiply by $t^q-1$. We obtain
\begin{align*}
\delta^{\tau}(p;q,\ell)=\dfrac{1}{2p}\;\sum_{t^p=1}\; \left[\sum_{i=1}^{\ell-1} 2t^{-i}+t^{-\ell}+1 \right]\left[\sum_{j=1}^{q^{\ast}\ell-1} (-1)^j 2t^{qj}+t^{\ell}+1\right].
\end{align*}
Introduce the notation 
\[
\sum_{t^p=1} t^a = \delta (a) = 
\begin{cases}
\; p, &\quad\text{if $a = 0 \pmod p$}, \\
\; 0, &\quad\text{otherwise}.
\end{cases}
\]
then
\begin{multline}\notag
\delta^{\tau}(p;q,\ell)=\dfrac{1}{2p}\;\sum_{i=1}^{\ell-1}\sum_{j=1}^{q^{\ast}\ell-1}4(-1)^j \delta(qj-i)+\dfrac{1}{p}\;\sum_{i=1}^{\ell-1}\delta(\ell-i) \\ +\dfrac{1}{p}\;\sum_{i=1}^{\ell-1}\delta(-i)+\dfrac{1}{2p}\;\sum_{j=1}^{q^{\ast}\ell-1}2(-1)^j \delta(qj-\ell)+\dfrac{1}{2p}\;\sum_{j=1}^{q^{\ast}\ell-1}2(-1)^j \delta(qj) \\ + \dfrac{1}{2p}\,\delta(-\ell) + \dfrac {1}{2p}\,\delta(\ell)+1.
\end{multline}
Since some of the sums in this formula are zero when $\ell \neq 0 \pmod p$ this further simplifies to
\begin{multline}\notag
\delta^{\tau}(p;q,\ell)  =\dfrac{2}{p}\;\sum_{i=1}^{\ell-1}\sum_{j=1}^{q^{\ast}\ell-1}(-1)^j \delta(qj-i) \\ +\dfrac{1}{p}\;\sum_{j=1}^{q^{\ast}\ell-1}(-1)^j \delta(qj-\ell)+\dfrac{1}{p}\;\sum_{j=1}^{q^{\ast}\ell-1}(-1)^j \delta(qj)+1.
\end{multline}

This formula has the following combinatorial interpretation. Consider the $(i,j)$-lattice with $i = 0, 1,\ldots, \ell$ and $j = 1,2, \ldots, q^{\ast}\ell-1$. The first sum is a signed count of the lattice points in the interior of the rectangle with $qj = i \pmod p$. The second sum picks up points on the line $i=\ell$ that satisfy $qj = \ell \pmod p$. Finally, the third sum counts points on the $j$-axis that satisfy $qj = 0 \pmod p $.
Plot the function $\phi(j)=qj/p-i/p$ as a family of lines for $i=0,\ldots,\ell$. The convex hull of these lines is the parallelogram $\mathcal{P}(p;q,\ell)$ with vertices $(0,0), (0,-\ell/p), (q^{\ast} \ell-1, q(q^{\ast}\ell-1)/p),(q^{\ast}\ell-1,q(q^{\ast}\ell-1)/p-\ell/p)$. In terms of this parallelogram, the outcome of the above calculation can be stated as follows. 
\begin{theorem}
Let $p$ and $q$ denote non-zero, relatively prime positive integers and $\ell$ an even integer such that $\ell \ne 0\pmod p$. Then
\begin{align*}
\delta^{\tau}(p;q,\ell)=\#^s\; \mathcal{P}(p;q,\ell)+1
\end{align*}
where $\#^s$ denotes the following signed count of lattice points in $\mathcal{P}(p;q,\ell)$: the sign of a lattice point $(j,\phi(j))$ is given by $(-1)^j$, the origin is not counted, and points in the interior are counted with weight $2$.
\end{theorem}
Due to the weight of interior points and symmetry on the boundary, the equivariant $\rho$-invariant $\delta^{\tau}(p;q,\ell)$ is always an odd integer. A similar description could be obtained for the other formula of Theorem A but we will not need this in the current paper.
\begin{remark}
It should be mentioned that Casson and Gordon \cite{casson1986cobordism} have a similar formula for their function $\delta(p;q,\ell)$ which reads 
\[
\delta(p;q,\ell) = -\dfrac{2q^{\ast}\ell^2}{p}+\#\, \mathcal{P}(p;q,\ell)+1.
\]
\end{remark}
\section{Brieskorn Homology Spheres}
Let $p$ and $q$ be positive relatively prime odd integers. The Brieskorn homology sphere $\Sigma = \Sigma(2,p,q)$ is the link of singularity at the origin of the complex polynomial $x^2 + y^p + z^q = 0$, in particular, it is canonically oriented. The involution $\tau(x,y,z) = (-x,y,z)$ makes $\Sigma$ into a double branched cover of $S^3$ with branch set the right-handed $(p,q)$ torus knot. The manifold $\Sigma$ is Seifert fibered over $S^2$ with a set of unnormalized Seifert invariants $(2,b_1)$, $(p,b_2)$, $(q,b_3)$ with even $b_2$ and $b_3$ found from the equation $pqb_1 + 2qb_2 + 2pb_3 = 1$. For a choice of base point fixed by $\tau$, the fundamental group of $\Sigma$ is isomorphic to
\[
\left< x_1,x_2,x_3,h, \vv [x_i,h]=1, x_1^2=h^{-b_1}, x_2^p=h^{-b_2},x_3^q=h^{-b_3}, x_1 x_2 x_3=1\right>
\]
and the induced action $\tau_*: \pi_1 (\Sigma) \to \pi_1 (\Sigma)$ is given by the formula $\tau_{\ast}(g) = x_1^{-1} g\, x_1$, see for instance \cite[Proposition 2.4]{collin1999geometric}. 

Since $\Sigma$ is an integral homology sphere, all reducible $SU(2)$ representations of $\pi_1 (\Sigma)$ are trivial. Let $\alpha: \pi_1(\Sigma) \to SU(2)$ be an irreducible representation then $\alpha(h) = -1$, therefore, $\tau^*\alpha = u\alpha u^{-1}$, where $u = \alpha (x_1)$ is an element of order four. Note for future use that up to conjugation $u = i \in SU(2)$ and $\Ad u = \diag (1,-1,-1) \in SO(3)$. The flat $SU(2)$ bundle $E_{\alpha} \to \Sigma$ is trivial for topological reasons, and $\tau$ admits a lift $\nu: E_{\alpha} \to E_{\alpha}$ which is unique up to gauge equivalence. The adjoint representation $\Ad\alpha: \pi_1(\Sigma) \to SO(3)$ gives rise to a flat bundle $E_{\Ad\alpha} = \Ad (E_{\alpha})$ and to a constant lift of order two called again $\nu$. We wish to calculate the equivariant $\rho$--invariant $\rho_{\Ad\alpha}(\nu,\Sigma)$.

Our calculation will rely on the flat cobordism technique of Fintushel and Stern \cite{FS85}. Let $W$ be the mapping cylinder of the quotient map $\Sigma \to S^2$. It is well known that $W$ is an orbifold with boundary $\Sigma$ and with singularities that are cones over the lens spaces $L(2,1)$, $L(p,b_2)$, and $L(q,b_3)$. The orbifold $W$ can also be viewed as the orbit space of the circle action on $\Sigma\times D^2$ given by the Seifert fibered structure on $\Sigma$ and by the complex multiplication on the 2-disk. We orient $W$ so that the induced orientation on its boundary matches the canonical orientation on $\Sigma$. With this orientation convention, the intersection form on $H_2 (W;\mathbb R) = \mathbb R$ is negative definite.

The involution $\tau$ naturally extends to an involution $\tau: W \to W$. To figure out how $\tau$ acts on the lens spaces, we consider equivariant tubes $D^2 \times S^1 \times D^2$ around the orbits of finite isotropy in $\Sigma \times D^2$. The circle actions on these tubes were described by Fintushel and Stern \cite{FS85}. For the singular orbit of isotropy 2, the circle action is given by $t(z,s,w) = (t^{pq} z,t^2 s, tw)$. The involution $\tau$, which corresponds to the choice of $t = -1$, acts as $\tau (z,s) = (-z,s)$ and hence is of type $g_2$. On the singular orbits of isotropies $p$ and $q$ the circle actions are given by $t(z,s,w) = (t^{2q} z, t^p s, tw)$ and $t(z,s,w) = (t^{2p} z, t^q s, tw)$, respectively. In both cases, specifying $t = -1$ yields the involution $\tau(z,s) = (z, -s)$ of type $g_3$. Therefore, the induced involution is of type $\textbf{A}$ on $L(2,1) = \bR P^3$ and of type $\textbf{B}$ on the lens spaces $L(p,b_2)$ and $L(q,b_3)$. Note that the involution $\tau$ makes $L(2,1)$ into a double branched cover of $S^3$ with branch set a Hopf link.

The manifold $W_0$ obtained from $W$ by chopping off $\tau$-invariant open cones over the lens spaces is an oriented cobordism between the lens spaces and $\Sigma$. The fixed point set $W_0^{\tau}$ of the involution $\tau: W_0 \to W_0$ consists of two connected components. The first component $F_1$ is a cylinder which interpolates between the circle of fixed points in $\Sigma$ and one of the two circles in the fixed point set in $L(2,1)$, while the second component $F_2$ is the orbit 2-sphere with three disks removed so that each of the three resulting boundary circles lies in a respective lens space. 

Let us choose a base point $x_0$ in the fixed point set of the involution $\tau: \Sigma \to \Sigma$ and observe that
\[
\pi_1 W_0 = \left< x_1,x_2,x_3\vv x_1^2=x^p_2=x_3^q=1,x_1x_2x_3=1 \right> = \pi_1\Sigma/\langle h \rangle.
\]
Since $\Ad\alpha: \pi_1 (\Sigma) \to SO(3)$ sends the central element $h$ to identity, the manifold $W_0$ is a flat $SO(3)$ cobordism. The representation $\Ad\alpha: \pi_1 W_0 \to SO(3)$ gives rise to a flat bundle over $W_0$ which is no longer trivial, and the involution $\tau: W_0 \to W_0$ admits a lift $\nu$ to that bundle.
\begin{lemma}
This lift restricts to a trivial lift over each of the lens spaces $L(p,b_2)$ and $L(q,b_3)$. 
\end{lemma}
\begin{proof}
Note that all $SO(3)$ bundles over $L(p,b_2)$ and $L(q,b_3)$ are trivial for topological reasons because $p$ and $q$ are odd, therefore, it makes sense to talk about trivial lifts. We will prove the triviality of the lift  for $L(p,b_2)$, the argument for $L(q,b_3)$ is similar. The action on the fundamental group $\tau_{\ast}:\pi_1 (W_0,x_0) \rightarrow \pi_1(W_0,x_0)$ satisfies $\tau_{\ast}(x_2)=x_1 x_2 x_1^{-1}$. Let $y_0$ be a base point in $L(p,b_2)$ on the fixed point circle. We are going to determine the induced action $\tau_{\ast}:\pi_1(L(p,b_2),y_0)\rightarrow \pi_1(L(p,b_2),y_0)$. To do this, choose a generator $t_2$ in $\pi_1(L(p,b_2),y_0)$ and a path $\gamma$ which conjugates the generators $\gamma t_2 \gamma^{-1}=x_2$ from the base points $x_0$ to $y_0$. This path is chosen to follow along the fixed component $F_1$ from $x_0$  to $L(2,1)$ where it lies in one component of the link. From there, follow a straight line to join with other fixed circle component in $L(2,1)$ and finally from $L(2,1)$ to $L(p,b_2)$ along the second fixed component $F_2$. We compute the action as follows:
\begin{align*}
\tau_{\ast}(t_2)&=\tau_{\ast}(\gamma)^{-1}\tau_{\ast}(x_2)\tau_{\ast}(\gamma) \\
&=\tau_{\ast}(\gamma)^{-1}(x_1 x_2 x_1^{-1})\tau_{\ast}(\gamma) \\
&=u t_2 u^{-1}  
\end{align*}
where $u=\tau_{\ast}(\gamma)^{-1} x_1 \gamma$. For any irreducible $SO(3)$ representation $\Ad \alpha$ of $\pi_1(W_0,x_0)$, the restriction of the representation to the lens space $L(p,b_2)$ satisfies\begin{align*}
\tau^{\ast}(i^{\ast}\Ad \alpha)=(\gamma u \gamma^{-1})i^{\ast}(\Ad \alpha)(\gamma u^{-1}\gamma^{-1}).
\end{align*} 
To determine the lift we must conjugate $u$ by $\gamma$ and we obtain $\gamma u \gamma^{-1}=(\gamma \tau_{\ast}(\gamma)^{-1})x_1$, which can be shown to be homotopically trivial.  

To see this, note the restriction of the involution $\tau$ to $L(2,1)$ can be given by $[x_1:x_2:x_3:x_4] \mapsto [x_1:-x_2:x_3,-x_4]$ in homogeneous coordinates. Let $c(t)=[1-t:0:0:t]$ be the path in $L(2,1)$ joining the two fixed points $[1:0:0:0]$ and $[0:0:0:1]$ in the two component link in $L(2,1)$. The image of this path under the involution is $(\tau c)(t)=[1-t:0:0:-t]$ and its inverse is given by $(\tau c)^{-1}(t)=[t:0:0:t-1]$. Now the path 
\begin{align*}
\sigma(t)=
\begin{cases} 
      c(2t)  & 0 \leq t \leq 1/2 \\
      (\tau c)^{-1}(2t-1) & 1/2 \leq t \leq 1 \\
\end{cases}
\end{align*}
is the generator $x_1$ of $\pi_1(L(2,1))$ since its lift to $S^3$ joins two antipodal points $(0,0,0,1)$ and $(0,0,0,-1)$. It follows from this that $(\gamma \tau_{\ast}(\gamma)^{-1})x_1=x_1^2=1$.
\end{proof}
Let $\alpha: \pi_1(\Sigma) \to SU(2)$ be an irreducible representation and apply the $G$-signature theorem of Donnelly \cite{Don78} to the flat cobordism $W_0$ to obtain 
\begin{multline}\notag
\sign_{\Ad \alpha}(\nu, W_0)=\int_{W_0^{\tau}} \ch_{\nu}(V_{\alpha})e(N(W_0^{\tau})) - \eta_{\Ad \alpha}(\nu,\Sigma) \\ + \eta_{\Ad \alpha_1}(\nu,L(2,1)) + \eta_{\Ad \alpha_2}(\tau, L(p,b_2)) + \eta_{\Ad \alpha_3}(\tau,L(q,b_3)),
\end{multline}
where $V_{\alpha} = E_{\Ad\alpha}\,\otimes\,\mathbb C$ is a flat $SU(3)$ bundle and $\Ad\alpha_i$ stand for the induced representations on the fundamental groups of the lens spaces. The fixed point set $W_0^{\tau}$ splits into a disjoint union of surfaces $F_1$ and $F_2$ with $\ch_{\nu} (V_{\alpha}) = -1$ over $F_1$ and $\ch_{\nu}(V_{\alpha}) = \ch_{\tau}(V_{\alpha}) = 3$ over $F_2$. This reduces the above formula to 
\begin{multline}\notag
\sign_{\Ad \alpha} (\nu, W_0) = -\int_{F_1} e(N(F_1)) + 3\int_{F_2} e(N(F_2))  - \eta_{\Ad \alpha}(\nu,\Sigma) \\ + \eta_{\Ad\alpha_1}(\nu,L(2,1)) + \eta_{\Ad\alpha_2}(\tau, L(p,b_2)) + \eta_{\Ad \alpha_3}(\tau,L(q,b_3)).
\end{multline}
Similarly, for the trivial representation $\Ad\theta: \pi_1 (\Sigma) \to SO(3)$ and a constant lift $\nu$ of order two we have
\begin{multline}\notag
\sign_{\Ad \theta} (\nu, W_0) = -\int_{F_1} e(N(F_1)) - \int_{F_2} e(N(F_2)) - \eta_{\Ad \theta}(\nu,\Sigma) \\ + \eta_{\Ad \theta}(\nu,L(2,1)) +\eta_{\Ad \theta}(\nu,L(p,b_2))+\eta_{\Ad \theta}(\nu,L(q,b_3)).
\end{multline}
Taking the difference of the above two formulas we obtain
\begin{multline}\notag
\sign_{\Ad \alpha}(\nu, W_{0})- \sign_{\Ad \theta}(\nu,W_0) = 4 \int_{F_2} e(N(F_2)) \\  - \rho_{\Ad \alpha}(\nu,\Sigma)+\rho_{\Ad \alpha_1}(\nu,L(2,1)) +\eta_{\Ad \alpha_2}(\tau,L(p,b_2)) \\ +\eta_{\Ad \alpha_3}(\tau,L(q,b_3))-\eta_{\Ad \theta}(\nu,L(p,b_2))-\eta_{\Ad \theta}(\nu,L(q,b_3)).
\end{multline}
We showed in Section \ref{S:complex} that $\rho_{\Ad \alpha_1}(\nu,L(2,1)) = 0$. The equivariant $\eta$-inva\-riants in the above formula do not add up to equivariant $\rho$-invariants because of the discrepancy in the lifts. To rectify this problem, note that 
\[
\eta_{\Ad \theta}(\nu,L(p,b_2)) = -\eta(\tau,L(p,b_2)),\quad
\eta_{\Ad \theta}(\tau,L(p,b_2)) = 3\eta(\tau,L(p,b_2)),
\]
where $\eta(\tau,L(p,b_2))$ stands for the equivariant $\eta$-invariant corresponding to the trivial $U(1)$ representation, and that similar formulas hold for $L(q,b_3)$. Therefore, we can write
\begin{multline}\notag
\sign_{\Ad \alpha}(\nu, W_{0}) - \sign_{\Ad \theta}(\nu,W_0) = 4 \int_{F_2} e(N(F_2)) \\ - \rho_{\Ad \alpha}(\nu,\Sigma)+\rho_{\Ad \alpha_2}(\tau,L(p,b_2)) + \rho_{\Ad \alpha_3}(\tau,L(q,b_3)) \\ + 4\eta(\tau,L(p,b_2)) + 4\eta(\tau,L(q,b_3)).
\end{multline}
Since the involution $\tau$ is part of the circle action on $\Sigma$, the twisted signature vanishes and $\sign_{\Ad\theta}(\nu,W_0) = -\sign(W_0) = 1$. The integral term in the above formula is the orbifold self-intersection number of the orbit $2$-sphere, which equals $-1/2pq$. Therefore,
\begin{multline}\notag
\rho_{\Ad \alpha}(\nu,\Sigma) = 1- \dfrac{2}{pq}+\rho_{\Ad \alpha_2}(\tau,L(p,b_2)) + \rho_{\Ad \alpha_3}(\tau, L(q,b_3))\\+4\eta(\tau,L(p,b_2)) + 4\eta(\tau,L(q,b_3)) 
\end{multline}
with
\[
\eta(\tau,L(p,b_2)) = -\dfrac{1}{p}\; \sum_{k=1}^{p-1}\; \cot \left( \dfrac{\pi k}{p} \right)\cot\left( \dfrac{\pi b_2 k}{p} + \dfrac {\pi}{2}\right)\;
\]
and  
\[
\eta(\tau,L(q,b_3)) = -\dfrac{1}{p}\; \sum_{k=1}^{q-1}\; \cot \left( \dfrac{\pi k}{q} \right)\cot\left( \dfrac{\pi b_3 k}{q} + \dfrac {\pi}{2}\right).
\]
\smallskip\noindent

To finish the calculation, express the equivariant $\rho$-invariants in terms of $U(1)$ representations,
\begin{gather}
\rho_{\Ad \alpha_2}(\tau,L(p,b_2)) = \eta_{\alpha_2}(\tau,L(p,b_2)) + \eta_{\,\overline{\alpha}_2}(\tau,L(p,b_2))-2\eta(\tau,L(p,b_2),\notag \\
\rho_{\Ad \alpha_3}(\tau,L(q,b_3)) =  \eta_{\alpha_3}(\tau,L(q,b_3)) + \eta_{\,\overline{\alpha}_3}(\tau,L(q,b_3))-2\eta(\tau,L(q,b_3),\notag
\end{gather}
and use the formulas of Theorem A to arrive at the following result. 
\begin{theorem}
Let $p,q$ be relatively prime odd integers and $\alpha: \pi_1 (\Sigma(2,p,q)) \to SU(2)$ an irreducible representation given by its rotation numbers $\ell_1$, $\ell_2$ and $\ell_3$ as in Fintushel--Stern \cite{FS90}. Let $\tau: \Sigma(2,p,q) \to \Sigma(2,p,q)$ be the involution contained in the natural circle action on $\Sigma(2,p,q)$, and let $\nu$ be the unique lift of $\tau$ to the flat bundle $E_{\Ad\alpha}$. The equivariant $\rho$-invariant of $\Ad\alpha$ with respect to this lift is given by 
\begin{align*}
\rho_{\Ad \alpha}(\nu,\Sigma(2,p,q)) = 1-\dfrac{2}{pq} &-\dfrac{4}{p}\;\sum_{k=1}^{p-1}\cot\left(\dfrac{\pi k}{p}\right)\tan\left(\dfrac{\pi b_2 k}{p}\right)\cos^{2}\left(\dfrac{\pi k l_2}{p}\right) \\&-\dfrac{4}{q}\;\sum_{k=1}^{q-1}\cot\left(\dfrac{\pi k}{q}\right)\tan\left(\dfrac{\pi b_3 k}{q}\right)\cos^{2}\left(\dfrac{\pi k l_3}{q}\right).
\end{align*}  
\end{theorem}


\section{On instanton homology of torus knots}\label{S:floer}
Kronheimer and Mrowka \cite{kronheimer2011khovanov} defined singular instanton Floer homology $I^{\natural}(K)$ for knots $K$ in the 3-sphere and used it to prove that Khovanov homology is an unknot detector. The chain complex $IC^{\natural}(K)$ for computing this Floer homology was described in \cite{kronheimer2011khovanov} in terms of orbifold gauge theory. Poudel and Saveliev \cite{poudel2015link} gave an alternative description of $IC^{\natural}(K)$ in terms of equivariant gauge theory on the double branched cover $Y \to S^3$ with branch set $K$. The generators of their chain complex arise from equivariant representations $\alpha: \pi_1 (Y) \to SO(3)$, and their gradings are given by the equivariant index of the ASD operator $\cD_A\,(\alpha,\theta)$ on the infinite cylinder $\mathbb R \times Y$ twisted by a connection $A$ interpolating between $\alpha$ at minus infinity and $\theta$ at plus infinity. In this section, we will use this description together with the formulas of Theorem B to compute the generators and gradings in the chain complexes $IC^{\natural}(K)$ of all $(p,q)$-torus knots $K$ with odd $p$ and $q$.

Let $K$ be a right-handed $(p,q)$-torus knot with odd relatively prime $p$ and $q$. The double branched cover of $S^3$ with branch set the knot $K$ is the Brieskorn homology sphere $\Sigma(2,p,q)$. Since the covering translation $\tau$ is part of the Seifert fibered structure on $\Sigma(2,p,q)$, it is homotopic to the identity hence all representations $\alpha: \pi_1 (\Sigma(2,p,q)) \to SO(3)$ are equivariant. According to \cite{poudel2015link}, the trivial representation $\theta$ gives rise to a single generator in $IC^{\natural}(\Sigma(2,p,q))$ of grading $\sign K = 0 \pmod 4$, and every irreducible representation $\alpha: \pi_1 (Y) \to SO(3)$ gives rise to four generators, two of grading $\mu(\alpha)\pmod 4$ and two of grading $\mu(\alpha) + 1 \pmod 4$, where
\[
\mu(\alpha) =\dfrac{1}{2}\cdot \gr(\alpha)+\frac 1 4\cdot(1-\rho_{\Ad \alpha}(\nu,\Sigma(2,p,q))
\]
and $\gr(\alpha)$ is the grading of $\alpha$ as a generator in the instanton Floer chain complex of the homology sphere $\Sigma(2,p,q)$, see \cite[Section 6.2]{poudel2015link}. 

An explicit formula for the gradings $\gr(\alpha)$ can be found in Fintushel--Stern \cite{FS90}. Let $\alpha$ be given by its rotation numbers $\ell_1=1$, $\ell_2$ and $\ell_3$, and define $e=pq\ell_1+2q\ell_2+2p\ell_3$. Then 
\begin{multline}
\gr(\alpha) = \dfrac{e^2}{pq}+\dfrac{2}{p}\, \sum_{k=1}^{p-1}\cot\left(\dfrac{\pi k}{p}\right)\cot\left(\dfrac{\pi b_2 k}{p}\right)\sin^{2}\left(\dfrac{\pi k \ell_2}{p}\right) \\ +\dfrac{2}{q}\, \sum_{k=1}^{q-1}\cot\left(\dfrac{\pi k}{q}\right)\cot\left(\dfrac{\pi b_3 k}{q}\right)\sin^{2}\left(\dfrac{\pi k \ell_3}{q}\right).
\end{multline}
Combining this formula with the formula of Theorem B for the equivariant $\rho$-invariant we obtain
\begin{align*}
\mu(\alpha)=\dfrac{e^2+1}{2pq}\,+& \,\dfrac{1}{2}\left( \delta(p;b_2,\ell_2)+\delta(q;b_3,\ell_3)\right)\\+&\dfrac{1}{2}\left( \delta^{\tau}(p;b_2,\ell_2)+\delta^{\tau}(q,b_3,\ell_3) \right)-\dfrac{1}{2} \left( D(p;b_2)+D(q;b_3) \right),
\end{align*}
where 
\begin{align*}
D(p;b)=\dfrac{2}{p}\sum_{k=1}^{p-1}\cot\left( \dfrac{\pi k}{p} \right)\cot \left( \dfrac{\pi b k}{p}+\dfrac{\pi}{2} \right).
\end{align*}
\begin{remark}
Lemma 4.1 of Collin--Saveliev \cite{collin2001equivariant} shows that $\mu(\alpha) = \gr(\alpha)\pmod 2$ for all irreducible $\alpha$
\end{remark}
The terms $D(p;b)$ in the above formula are a generalization of Dedekind sums, which was studied by Dieter \cite{dieter1984cotangent}
\begin{align*}
c(q,p\vv x,z)=\dfrac{1}{p}\sum_{k=1}^{p-1} \cot\left( \dfrac{\pi q (k+z)}{p}-x\pi \right) \cot\left(\dfrac{\pi(k+z)}{p}\right)
\end{align*}
where $0 \leq x,z<1$ are real numbers, and $p,q$ are relatively prime positive integers. When $x=1/2$ and $z=0$ we obtain
\begin{align*}
c(q,p \vv 1/2, 0)&=\dfrac{1}{p} \sum_{k=1}^{p-1} \cot\left( \dfrac{\pi q k}{p}-\dfrac{\pi}{2} \right) \cot\left(\dfrac{\pi k}{p}\right)\\
&=\dfrac{1}{p} \sum_{k=1}^{p-1} \cot\left( \dfrac{\pi q k}{p}+\dfrac{\pi}{2} \right) \cot\left(\dfrac{\pi k}{p}\right)\\
&=\dfrac{1}{2}D(p;q).
\end{align*}
\smallskip
\begin{example}
Consider the $(3,7)$-torus knot $T_{3,7}\subset S^3$. Its double branched cover is the Brieskorn integral homology sphere $\Sigma(2,3,7)$ which has $2$ irreducible $SO(3)$ representations whose rotation numbers are $\alpha_1=(1,2,2)$ and $\alpha_2=(1,2,4)$ with instanton grading $gr(\alpha_1)\equiv 7 \pmod 8$ and $gr(\alpha_2)\equiv 3 \pmod 8$. Hence the instanton Floer homology groups of $\Sigma(2,3,7)$ are given by 
\begin{align*}
I_{\ast}(\Sigma(2,3,7))=(0,0,0,\bZ,0,0,0,\bZ).
\end{align*}  
Using the formula above, we can compute the gradings $\mu(\alpha_1)=87 \equiv 3 \pmod 4$ and $\mu(\alpha_2)=125\equiv 1 \pmod 4$. Taking into account the special generator of grading zero, we obtain 
\begin{align*}
IC^{\natural}(T_{3,7})&=(0,\bZ^2,\bZ^2,0)\oplus (\bZ^2,0,0,\bZ^2)\oplus (\bZ,0,0,0)
=(\bZ^3,\bZ^2,\bZ^2,\bZ^2).
\end{align*}
\end{example}
\begin{example}
We can include the previous example in an infinite family of torus knots $T_{3,6n+1}\subset S^3$. The double branched cover is the Brieskorn homology sphere $\Sigma(2,3,6n+1)$ with Seifert invariants $\Sigma((2,-3),(3,2),(6n+1,5n+1))$ if $n$ is odd and $\Sigma((2,-1),(3,2),(6n+1,-n))$ for $n$ even. There are $2n$ irreducible $SO(3)$ representations, up to conjugacy, each with rotation numbers of the form $(1,2,\ell_3)$ with $\ell_3$ an even integer on the interval $[n+1,5n]$. The instanton Floer homology is known to be given by (see \cite{FS90})
\begin{align*}
I_{\ast}(\Sigma(2,3,6n+1))= \begin{cases} (0,\bZ^{n/2},0,\bZ^{n/2},0,\bZ^{n/2},0,\bZ^{n/2}) & \text{if $n$ is even} \\ (0,\bZ^{\frac{n-1}{2}},0,\bZ^{\frac{n+1}{2}},0,\bZ^{\frac{n-1}{2}},0,\bZ^{\frac{n+1}{2}}) & \text{if $n$ is odd}.  \end{cases}
\end{align*}
The formula for the gradings in knot Floer homology is 
\begin{align*}
\mu(\alpha)=\dfrac{e^2+1}{36n+6}&+\dfrac{1}{2}\left(\delta(3;2,2)+\delta(q;b,\ell_3)\right) \\
 &+\dfrac{1}{2}\left(\delta^{\tau}(3;2,2)+\delta^{\tau}(q;b,\ell_3) \right)-\dfrac{1}{2}\left(D(3;2)+D(q;b)\right)
\end{align*}
where $q=6n+1$, $b=5n+1$ and $e=7(6n+1)+6\ell_3$. The equalities 
\begin{align*}
\delta(3;2,2)=- \dfrac{1}{3}\, , \quad  \delta^{\tau}(3;2,2)=1 \, , \quad  D(3;2)=\dfrac{4}{3}
\end{align*}
are readily verified. The other terms require more work and we use a result of \cite{dieter1984cotangent} to evaluate the generalized Dedekind sum
\begin{align*}
D(6n+1;5n+1)=2 \, c(5n+1;6n+1 \vv 1/2,0).
\end{align*}
For $n > 1$ define a sequence of integers
\begin{align*}
a_{i+1}=a_{i-1}-a_iq_i
\end{align*}
with $a_0=6n+1$, $a_1=5n+1$ and $q_i=\left[\dfrac{a_{i-1}}{a_i}\right]$. Similarly, set $s_{-1}=0$, $s_0=1$, $x_0=0$, $x_1=1/2$ and 
\begin{align*}
s_i=s_{i-1}q_i+s_{i-2} \\
x_{i+1}=x_{i-1}-x_iq_i.
\end{align*} 
This defines sequences $[a_0,a_1,a_2,a_3,a_4]=[6n+1,5n+1,n,1,0]$,$[q_1,q_2,q_3]=[1,5,n]$, $[s_1,s_2,s_3]=[1,6,6n+1]$, $[x_0,x_1,x_2,x_3,x_4]=[0,1/2,-1/2,3,-1/2-3n]$, which according to \cite[Theorem 3.2]{dieter1984cotangent} evaluates  
\begin{align*}
D(6n+1;5n+1)&=\dfrac{2}{6n+1}\sum_{k=1}^{6n} \cot\left(\dfrac{\pi k}{6n+1}\right)\cot\left(\dfrac{\pi k (5n+1)}{6n+1}+\dfrac{\pi}{2}\right)\\
&= 2\left( n-1+\dfrac{5n+1}{6n+1} \right).
\end{align*}
Using these values we obtain 
\begin{align*}
\mu(\alpha)=\dfrac{2(144n^2+48n+42n\ell_3+4+7\ell_3+3\ell_3^2)}{6n+1}+\dfrac{1}{2}\left( \delta^{\tau}(q;b,\ell_3)+\delta(q;b,\ell_3) \right).
\end{align*}
This further simplifies by using the formulas 
\begin{align*}
\delta(q;b,\ell_3)&=\dfrac{-2b^{\ast}\ell_3^2}{q}+\#P(q;b,\ell_3)+1 \quad \text{and} \\
\delta^{\tau}(q;b,\ell_3)&=\#P^s(q;b,\ell_3)+1 \, ,
\end{align*}
where $q=6n+1$, $b=5n+1$ and $b^{\ast}=6$ the multiplicative inverse of $b$ mod $q$. The formula for the grading reduces to 
\begin{align*}
\mu(\alpha&)=48n+14\ell_3+9+\dfrac{1}{2}\left( \#P^s(q;b)+\#P(q;b)\right)\\
&\equiv 1+ \dfrac{1}{2}\left( \#P^s(q;b)\, + \, \#P(q;b)\right) \pmod 4
\end{align*}
which amounts to the combinatorial problem of counting lattice points in a parallelogram with appropriate weight and signs. To find a closed form formula, we take an alternative approach by using the identity
\begin{align*}
\delta^{\tau}(q;b,\ell_3)+\delta(q;b,\ell_3)=2\, \delta(q;2b,\ell_3)
\end{align*}
and the formula of Lawson \cite[p. 188]{Lawson87},
\begin{align*}
\mu(\alpha)&=\dfrac{2(144n^2+48n+42n\ell_3+4+7\ell_3+3\ell_3^2)}{6n+1}+\delta(q;2b,\ell_3)\\
&=48n+14\ell_3+8+N(q;2b,\ell_3)\\
&\equiv N(q;2b,\ell_3) \pmod 4\\
&\equiv -1-2\sum_{k=1}^{\ell_3-1} \left[ q;2b,k+1 \right] \pmod 4
\end{align*}
where $[x;y,z]$ is defined to be $1$ if one of $y,2y,...,(y^{\ast}-1)y \equiv z \pmod x$ and zero otherwise.  Since $2b\equiv 4n+1$ and $(2b)^{\ast} \equiv 3 \mod (6n+1)$, we obtain
\begin{align*}
\mu(\alpha) & \equiv -1-2\sum_{k=1}^{\ell_3-1} \left[6n+1;4n+1,k+1\right] \\
&= \begin{cases}  -1   \, , \quad \quad 2\leq \ell_3 \leq 2n \\
-3  \, , \quad \quad 2n+2 \leq \ell_3 \leq 4n\\
-5 \, , \quad \quad 4n+2 \leq \ell_3 \leq 5n.
\end{cases}
\end{align*}
If $n$ is odd we have 
\begin{align*}
\mu(\alpha)\equiv \begin{cases} 3 \, , \quad  \quad \ell_3=n+1,..2n \quad \text{and} \quad 4n+2,...,5n-1 \\
1 \, , \quad \quad \ell_3=2n+2,...,4n
\end{cases}
\end{align*}
and when $n$ is even,
\begin{align*}
\mu(\alpha)\equiv \begin{cases} 3 \, , \quad  \quad \ell_3=n+2,..2n \quad \text{and} \quad 4n+2,...,5n \\
1 \, , \quad \quad \ell_3=2n+2,...,4n.
\end{cases}
\end{align*}
In either case, half the representations have grading $1$ and the other half grading $3$. Taking the special generator into account in grading zero, we obtain 
\begin{align*}
IC^{\natural}(T_{3,6n+1})&=(\bZ^{2n+1},\bZ^{2n},\bZ^{2n},\bZ^{2n}).
\end{align*}
\end{example} All calculations agree with those of Hedden--Herald--Kirk \cite{hedden2014pillowcase} and support the conjecture about $IC^{\natural}(K)$ for torus knots $K$ stated in \cite[7.4 p. 48]{poudel2015link}.

\bibliographystyle{alpha}
\bibliography{Bibliography}
\end{document}